\documentclass[12pt,a4paper,final]{amsart}
\usepackage{amssymb,graphicx,color,esint,stmaryrd,hyperref}

\newcommand{\R}{{\mathbb R}}

\newcommand{\vol}{{\operatorname{vol}}}
\newcommand{\Z}{{\mathcal Z}}
\newcommand{\A}{{\mathcal A}}

\newtheorem{thm}{Theorem}
\newtheorem*{teorema}{Theorem}
\newtheorem{prop}{Proposition}
\newtheorem{lemma}{Lemma}
\newtheorem{cor}{Corollary}

\theoremstyle{definition}
\newtheorem{defini}{Definition}

\theoremstyle{remark}
\newtheorem{remark}{Remark}

\title[Carleson measures and Logvinenko-Sereda sets]{Carleson measures and
Logvinenko-Sereda
sets on compact manifolds}
\author{Joaquim Ortega-Cerd\`a}
\address{Dept.\ Matem\`atica Aplicada i An\`alisi,
 Universitat  de Barcelona,
Gran Via 585, 08007 Bar\-ce\-lo\-na, Spain}
\email{jortega@ub.edu}
\author{Bharti Pridhnani}
\address{Dept.\ Matem\`atica Aplicada i An\`alisi,
 Universitat  de Barcelona,
Gran Via 585, 08007 Bar\-ce\-lo\-na, Spain}
\email{bharti.pridhnani@ub.edu}
\thanks{Supported by the project MTM2008-05561-C02-01 and the CIRIT grant
2005SGR00611}
\date{\today}

\begin{document}

\begin{abstract}
Given a compact Riemannian manifold $M$ of dimension $m\geq 2$, we study the
space of functions of $L^2(M)$ generated by eigenfunctions of eigenvalues less
than $L\geq 1$ associated to the Laplace-Beltrami operator on $M$. On these
spaces we give a characterization of the Carleson measures and the
Logvinenko-Sereda sets.
\end{abstract}
\maketitle

\section*{Introduction and statement of the results}
Let $(M,g)$ be a smooth, connected, compact Riemannian manifold without
boundary of dimension $m\geq 2$. Let $dV$ be the volume element of $M$
associated to the metric $g_{ij}$. Let $\Delta_{M}$ be the Laplacian on $M$
associated to the metric $g_{ij}$. It is given in local coordinates by
\[
\Delta_{M}f=\frac{1}{\sqrt{|g|}}\sum_{i,j}\frac{\partial}{\partial
x_{i}}\left(\sqrt{|g|}g^{ij}\frac{\partial f}{\partial
x_{j}}\right),
\]
where $|g|=det(g_{ij})$ and $(g^{ij})_{ij}$ is the inverse matrix
of $(g_{ij})_{ij}$. Since $M$ is compact, $g_{ij}$ and all its
derivatives are bounded and we assume that the metric
$g$ is non-singular at each point of $M$.\\
\\
Since $M$ is compact, the spectrum of the Laplacian is discrete
and there is a sequence of eigenvalues
\[
0\leq \lambda_{1}\leq\lambda_{2}\leq\cdots \to \infty
\]
and an orthonormal basis $\phi_{i}$ of smooth real eigenfunctions of the
Laplacian i.e. $\Delta_{M}\phi_{i}=-\lambda_{i}\phi_{i}$. So $L^{2}(M)$ 
decomposes into an orthogonal direct sum of eigenfunctions of the Laplacian.\\
\\
We consider the following spaces of $L^{2}(M)$.
\[
E_{L}=\left\{f\in L^{2}(M)\ :\
f=\sum^{k_{L}}_{i=1}\beta_{i}\phi_{i},\
\Delta_{M}\phi_{i}=-\lambda_{i}\phi_{i},\ \lambda_{k_{L}}\leq
L\right\},
\]
where $L\geq 1$ and $k_{L}=\dim E_{L}$. $E_{L}$ is the space of $L^{2}(M)$
generated by eigenfunctions of eigenvalues $\lambda\leq L$. Thus in $E_L$ we
consider functions in $L^2(M)$ with a restriction on the support of its Fourier
transform. It is, in a sense, a Paley-Wiener type space on $M$ with bandwidth
$L$.\\
\\
The motivation of this paper is to show that the spaces $E_L$ behave like the
space defined in $\mathbb S^{d}$ ($d>1$) of spherical
harmonics of degree less than $\sqrt{L}$. In fact, the space $E_L$ is a generalization
of the spherical harmonics and the role of them are played by the
eigenfunctions. The cases $M=S^1$ and $M=\mathbb S^{d}$ ($d>1$) have been
studied in \cite{OrtSal} and \cite{Mar}, respectively. 

This similarity between eigenfunctions of the Laplacian and polynomials is not
new, for instance, Donnelly and Fefferman in \cite[Theorem 1]{DF}, showed that
on a compact manifold, an eigenfunction of eigenvalue $\lambda$
behaves essentially like a polynomial of degree $\sqrt{\lambda}$. 
In this direction they proved the following local $L^{\infty}$ Bernstein inequality.
\begin{teorema}[Donnelly-Fefferman]
Let $M$ be as above. If $u$ is an eigenfunction of the Laplacian
$\Delta_{M}u=-\lambda u$, then there exists $r_{0}=r_{0}(M)$ 
such that for all $r<r_{0}$ we have
\[
\max_{B(x,r)}|\nabla u|\leq \frac{C\lambda^{(m+2)/4}}{r}\max_{B(x,r)}|u|.
\]
\end{teorema}
The proof of the above estimate is rather delicate.
Donnelly and Fefferman conjectured that it is possible to replace
$\lambda^{(m+2)/4}$ by $\sqrt{\lambda}$ in the inequality. If the conjecture
holds, we have in particular, a global Bernstein type inequality:
\begin{equation}\label{bernsteintypeineq}
\left\|\nabla u\right\|_{\infty}\lesssim
\sqrt{\lambda}\left\|u\right\|_{\infty}.
\end{equation}
In fact, this weaker estimate holds and a proof will be given later. This fact
suggests that the right metric to study the space $E_L$ should be rescaled by
a factor $1/\sqrt{L}$ because in balls of radius $1/\sqrt{\lambda}$,
a bounded eigenfunction of eigenvalue $\lambda$ oscillates very little. 
\\
In the present work we will study for which measures
$\mu=\left\{\mu_{L}\right\}_L$ one has
\begin{equation}\label{aimtostudy}
\int_{M}|f|^2d\mu_{L}\approx \int_{M}|f|^2\, dV, \quad \forall f\in E_L
\end{equation}
with constants independent of $f$ and $L$. 

We will also study the inequality
\[
\int_{M}|f|^2d\mu_L\lesssim\int_{M}|f|^2\, dV
\]
that defines the Carleson measures and we will present a geometric
characterization of them. Inequality \eqref{aimtostudy} will be studied only for
the special case $d\mu_{L}=\chi_{A_{L}}dV$, where $\mathcal
A=\left\{A_{L}\right\}_L$ is a sequence of sets in the manifold. In such case,
when \eqref{aimtostudy} holds , we say that
$\mathcal A$ is a sequence of Logvinenko-Sereda sets. Our two main results are
the following:
\begin{thm}\label{thmCM}
Let $\mu=\left\{\mu_L\right\}_L$ be a sequence of measures on
$M$. Then $\mu$ is $L^{2}$-Carleson for $M$ if and only if
there exists a $C>0$ such that for all $L$
\[
\sup_{\xi\in M}\frac{\mu_{L}(B(\xi,1/\sqrt{L}))}{\vol(B(\xi,1/\sqrt{L}))}\leq C.
\]
\end{thm}
\begin{thm}\label{thmLS} The sequence of sets $\A=\left\{A_L\right\}_L$ is
Logvinenko-Sereda if and only if there is an $r>0$ such that
\[
\inf_{L}\inf_{z\in M}\frac{\vol(A_{L}\cap
B(z,r/\sqrt{L}))}{\vol(B(z,r/\sqrt{L}))}>0.
\]
\end{thm}
In what follows, when we write that $A\lesssim B$, $A\gtrsim B$ or $A\approx B$
we mean that there are constants depending only on the manifold such that $A\leq
C B$, $A\geq CB$ or $C_{1}B\leq A\leq C_{2}B$, respectively. Also, the value of
the constants appearing during a proof may change but they will be still denoted
with the same letter. We will denote by $B(\xi,r)$ a geodesic
ball in $M$ of center $\xi$ and radius $r$ and $\mathbb B(z,r)$
will denote an Euclidean ball in $\mathbb R^m$ of center $z$ and radius $r$.\\
\\
The structure of the paper is the following: in the first section, we will
explain the asymptotics of the reproducing kernel of the space $E_L$. In the
second section, we shall discuss one of the tools used: the harmonic extension
of functions in the space $E_L$. Following this, we will study the Carleson
measures associated to $M$ and we will prove Theorem \ref{thmCM}. In the last
section, Theorem \ref{thmLS} will be proved.\\
\\
\textbf{Acknowledgment.} We thank professor M. Sanch\'on for valuable
comments on the subject and professor S. Zelditch for providing us appropriate 
references.
\section{The reproducing kernel of $E_L$}
Let
\[
K_{L}(z,w):=\sum^{k_{L}}_{i=1}\phi_{i}(z)\phi_{i}(w)=\sum_{\lambda_{i}\leq
L}\phi_{i}(z)\phi_{i}(w).
\]
This function is the reproducing kernel of the space $E_L$, i.e.
\begin{align*}
\forall f\in E_L\quad f(z)=\langle f,K_L(z,\cdot)\rangle.
\end{align*}
Note that $\dim(E_L)=k_{L}=\#\left\{\lambda_{i}\leq L\right\}.$ The function
$K_{L}$ is also called the spectral function associated to the Laplacian.
H\"ormander in \cite{Horm}, proved the following estimates.
\begin{enumerate}
  \item $K_{L}(z,w)=O(L^{(m-1)/2}),\ \ z\neq w.$\\
  \item $K_{L}(z,z)=\frac{\sigma_{m}}{(2\pi)^{m}}L^{m/2}+O(L^{(m-1)/2})$
(uniformly in $z\in M$), where $\sigma_{m}=2\pi^{m/2}/(m\Gamma(m/2))$.\\
  \item $k_{L}=\frac{\vol(M)\sigma_{m}}{(2\pi)^{m}}L^{m/2}+O(L^{(m-1)/2})$.
\end{enumerate}
In fact, in \cite{Horm}, there are estimates for the
spectral function associated to any elliptic operator of order $n\geq 1$ with
constants depending only on the manifold.\\
\\
So, for $L$ big enough we have $k_{L}\approx L^{m/2}$ and
\[
\left\|K_{L}(z,\cdot)\right\|^{2}_{2}=K_{L}(z,z)\approx L^{m/2}\approx k_{L}
\]
with constants independent of $L$ and $z$.\\
\section{Harmonic extension}\label{harmonicextension}
In what follows, given $f\in E_{L}$, we will note by $h$ the
harmonic extension of $f$ in $N=M\times \R$. 
The metric that we consider in $N$ is the product metric, i.e. if we denote it
by
$\tilde{g}_{ij}$ for $i=1,\ldots,m+1$ then
\[
(\tilde{g}_{ij})_{i,j=1,\ldots,m+1}=\left(\begin{matrix}
	(g_{ij})^{m}_{i,j=1} & 0\\
	0 & 1
\end{matrix}\right).
\]
Using this matrix, we can calculate the gradient and the Laplacian for $N$. If
$h(z,t)$ is a function defined on $N$ then
\[
|\nabla_{N}h(z,t)|^2=|\nabla_{M}h(z,t)|^2+\left(\frac{\partial h}{\partial
t}(z,t)\right)^2
\]
and
\[
\Delta_{N}h(z,t)=\Delta_{M}h(z,t)+\frac{\partial^2 h}{\partial t^2}(z,t).
\]
Note that $|\nabla_{M}h(z,t)|\leq |\nabla_{N}h(z,t)|$.\\
\\
Let $f\in E_{L}$, we know that
\[
f=\sum^{k_{L}}_{i=1}\beta_{i}\phi_{i},\qquad
\Delta_{M}\phi_{i}=-\lambda_{i}\phi_{i},\qquad 0\leq \lambda_{i}\leq L.
\]
Define for $(z,t)\in N$
\[
h(z,t)=\sum^{k_{L}}_{i=1}\beta_{i}\phi_{i}(z)e^{\sqrt{\lambda_{i}}t}.
\]
Observe that $h(z,0)=f(z)$. Moreover $|\nabla_{M}f(z)|^2\le
|\nabla_{N}h(z,0)|^2.$\\

The function $h$ is harmonic in $N$ because
{\begin{align*}
\Delta_{N}h(z,t)&
=\sum^{k_{L}}_{i=1}\left[\beta_{i}e^{\sqrt{\lambda_{i}}t}\Delta_{M}\phi_{i}
(z)+\beta_{i}\phi_{i}(z)\Delta_{\R}(e^{\sqrt{\lambda_{i}}t})\right]=0.
\end{align*}
For the harmonic extension, we don't have the mean-value property, because it
is not true for all manifolds (only for the harmonic manifolds, see
\cite{Willmore} for a complete characterization of them). What is always true is
a ``submean value property" (with a uniform constant) for positive subharmonic
functions, see for example \cite[Chapter II, Section 6]{SchYau}).\\
\\
Observe that since $h$ is harmonic on $N$, $|h|^2$ is a positive subharmonic
function on $N$. In fact, $|h|^{p}$ is subharmonic for all $p\geq 1$ (for a
proof see \cite[Proposition 1]{GreWu}). Therefore, we know that for all
$r<R_{0}(M)$
\[
|h(z,t)|^{2}\lesssim \fint_{B(z,r/\sqrt{L})\times
I_{r}(t)}|h(w,s)|^{2}dV(w)ds,
\]
where $R_{0}(M)>0$ is the injectivity radius of $M$ 
and $I_{r}(t)=(t-r/\sqrt{L},t+r/\sqrt{L})$. In particular,
\begin{equation}\label{estf2}
|f(z)|^{2}\leq C_{r}L^{(m+1)/2}\int_{B(z,r/\sqrt{L})\times
I_{r}}|h(w,s)|^{2}dV(w)ds,
\end{equation}
where $I_{r}=I_{r}(0)$.\\
The following result relates the $L^2-$norm of $f$ and $h$.
\begin{prop}\label{normhandf}
Let $r>0$ fixed. If $f\in E_{L}$ then
\begin{equation}\label{desnormhandf}
2re^{-2r}\left\|f\right\|^{2}_{2}\leq
\sqrt{L}\left\|h\right\|^{2}_{L^{2}(M\times I_{r})}\leq
2re^{2r}\left\|f\right\|^{2}_{2}.
\end{equation}
Therefore, for $r<R_{0}(M)$
\[
\frac{\sqrt{L}}{2r}\left\|h\right\|^{2}_{L^{2}(M\times I_{r})}\approx
\left\|f\right\|^{2}_{2}
\]
with constants depending only on the manifold $M$.
\end{prop}
\begin{proof}
Using the orthogonality of $\left\{\phi_{i}\right\}_{i}$ we have
\begin{align*}
\left\|h\right\|^{2}_{L^{2}(M\times I_{r})} &
=\int_{I_{r}}\int_{M}\left|\sum^{k_{L}}_{i=1}\beta_{i}\phi_{i}(z)e^{\sqrt{
\lambda_{i}}t}\right|^{2}dV(z)dt\\
&=\int_{I_{r}}\sum^{k_{L}}_{i=1}\int_{M}|\beta_{i}|^{2}|\phi_{i}(z)|^{2}dV(z)e^
{2\sqrt{\lambda_{i}}t}dt\leq
\int_{I_{r}}e^{2\sqrt{L}t}dt\left\|f\right\|^{2}_{2}.
\end{align*}
Similarly, one can prove the left hand side inequality of 
\eqref{desnormhandf}.
\end{proof}

We recall now a result proved by Schoen and Yau that estimates the gradient 
of harmonic functions.

\begin{teorema}[Schoen-Yau]\label{YauGradestimate}
Let $N$ be a complete Riemannian manifold with Ricci curvature bounded below by
-(n-1)K ($n$ is the dimension of $N$ and $K$ a positive constant). Suppose
$B_{a}$ is a geodesic ball in $N$ with radius $a$ and $h$ is an harmonic
function on $B_{a}$. Then
\begin{equation}\label{SchYauestimate}
\sup_{B_{a/2}}|\nabla h|\leq
C_{n}\left(\frac{1+a\sqrt{K}}{a}\right)\sup_{B_{a}}|h|,
\end{equation}
where $C_{n}$ is a constant depending only on the dimension of $N$.
\end{teorema}
For a proof see \cite[Corollary 3.2., page 21]{SchYau}.
\begin{remark}
We will use Schoen and Yau's estimate in the following context. Take $N=M\times \R$. 
Observe that $Ricc(N)=Ricc(M)$ which is bounded below because $M$ is compact. 
Note that $N$ is complete because it is a product of two complete manifolds. We will take 
$a=r/\sqrt{L}$ ($r<R_{0}(M)$) and $B_{a}=B(z,r/\sqrt{L})\times I_{r}$ 
(note that this is not the ball of center $(z,0)\in N$ and radius $r/\sqrt{L}$, 
but it contains and it is contained in such ball of comparable radius). 
\end{remark}
Using Schoen and Yau's theorem, we deduce the global 
Bernstein inequality for a single eigenfunction.
\begin{cor}[Bernstein inequality]
If $u$ is an eigenfunction of eigenvalue $\lambda$, then
\begin{equation}\label{Bernsteinineqeigenfunction}
\left\|\nabla u\right\|_{\infty}\lesssim
\sqrt{\lambda}\left\|u\right\|_{\infty}.
\end{equation}
\end{cor}
\begin{proof}
The harmonic extension of $u$ is $h(z,t)=u(z)e^{\sqrt{\lambda} t}$. Applying 
inequality \eqref{SchYauestimate} to $h$ (taking $a=R_{0}(M)/(2\sqrt{\lambda})$),
\[
|\nabla u(z)|\lesssim \sqrt{\lambda}\left\|h\right\|_{L^{\infty}(M\times I_{R_{0}/2})}\approx 
\sqrt{\lambda}\left\|u\right\|_{\infty}.
\]
\end{proof}
We conjecture that in inequality \eqref{Bernsteinineqeigenfunction}, one 
can replace $u$ by the functions $f\in E_L$, i.e.
\[
\left\|\nabla f\right\|_{\infty}\lesssim \sqrt{L}\left\|f\right\|_{\infty}.
\]
For instance, as a direct 
consequence of Green's formula, we have the 
$L^2-$Bernstein inequality for the space $E_L$:
\[
\left\|\nabla f \right\|_{2}\lesssim \sqrt{L}\left\|f \right\|_2\quad \forall f\in E_L.
\]
For our purpose, it is sufficient a weaker Bernstein type inequality that 
compares the $L^{\infty}$ norm of the gradient with the $L^2$ 
norm of the function.
\begin{prop}\label{goodgradestimate}
Let $f\in E_{L}$. Then there exists a universal constant $C$ such that
\[
\left\|\nabla f\right\|_{\infty}\leq
C\sqrt{k_{L}}\sqrt{L}\left\|f\right\|_{2}.
\]
\end{prop}
For the proof, we need the following lemma.
\begin{lemma}\label{gradestimate}
For all $f\in E_{L}$ and $0<r<R_{0}(M)/2$,
\[
|\nabla f(z)|^{2}\leq C_{r}L^{(m+2+1)/2}\int_{B(z,r/\sqrt{L})\times
I_{r}}|h(w,s)|^{2}dV(w)ds.
\]
\end{lemma}
\begin{proof}
Using inequality \eqref{SchYauestimate} and the submean-value inequality 
for $|h|^2$, we have
\begin{align*}
|\nabla f(z)|^2 & \leq |\nabla h(z,0)|^2 \lesssim \frac{L}{r^{2}}
\sup_{B(z,r/\sqrt{L})\times I_{r}}|h(w,t)|^{2}\\
& \lesssim \frac{L^{(m+1+2)/2}}{\tilde{r}^{m+2+1}}\int_{B(z,\tilde{r}/\sqrt{L})\times
I_{\tilde{r}}}|h(\xi,s)|^{2}dV(\xi)ds,
\end{align*}
where $\tilde{r}=2r$.
\end{proof}
\begin{proof}[Proposition \ref{goodgradestimate}]
Using Lemma \ref{gradestimate}, given $0<r<R_{0}(M)/2$
there exists a constant $C_{r}$ such that
\begin{align*}
|\nabla f(z)|^{2} & \leq C_{r} k_{L}L\sqrt{L}\int_{M\times I_{r}}|h(w,s)|^{2}dV(w)ds
\overset{\text{Proposition
\ref{normhandf}}}{\approx}C_{r}k_{L}L\left\|f\right\|^{2}_{2}.
\end{align*}
Taking $r=R_{0}(M)/4$, we get that $|\nabla f(z)|^2\leq
Ck_{L}L\left\|f\right\|^2_{2}$ for all $z\in M$.
\end{proof}

\section{Characterization of Carleson measures}
\begin{defini}
Let $\mu=\left\{\mu_{L}\right\}_{L\geq 0}$ be a sequence of measures on $M$. We
say that $\mu$ is an $L^2-$\textbf{Carleson sequence} for $M$ if there exists a
positive constant $C$ such that for all $L$ and $f_{L}\in E_{L}$
\[
\int_{M}|f_{L}|^{2}d\mu_{L}\leq C\int_{M}|f_{L}|^{2}dV.
\]
\end{defini}
\begin{thm}\label{charactseqCarlmeas}
Let $\mu$ be a sequence of measures on $M$. Then, $\mu$ is $L^{2}-$Carleson for
$M$ if and only if there exists a $C>0$ such that for all $L$
\begin{equation}\label{condCarlmeas}
\sup_{\xi\in M}\mu_{L}(B(\xi,1/\sqrt{L}))\leq \frac{C}{k_{L}}.
\end{equation}
\end{thm}
\begin{remark}
Condition \eqref{condCarlmeas} can be viewed as
\[
\sup_{\xi\in
M}\frac{\mu_{L}(B(\xi,1/\sqrt{L}))}{\vol(B(\xi,1/\sqrt{L}))}\lesssim 1.
\]
\end{remark}
First, we prove the following result.
\begin{lemma}
Let $\mu$ be a sequence of measures on $M$. Then, the following conditions are
equivalent.
\begin{enumerate}
  \item There exists a constant $C=C(M)>0$ such that for each $L$
\[
\sup_{\xi\in M}\mu_{L}(B(\xi,1/\sqrt{L}))\leq \frac{C}{k_{L}}.
\]
  \item There exist $c=c(M)>0$ ($c<1$ small) and $C=C(M)>0$ such that for all
$L$
\[
\sup_{\xi\in M}\mu_{L}(B(\xi,c/\sqrt{L}))\leq \frac{C}{k_{L}}.
\]
\end{enumerate}
\end{lemma}
\begin{proof}
Obviously, the first condition implies the second one since
\[
B(\xi, c/\sqrt{L})\subset B(\xi, 1/\sqrt{L}).
\]
Let's see the converse. The manifold $M$ is covered by the union of balls of
center $\xi\in M$ and radius $c/\sqrt{L}$. Taking into account the 5-$r$
covering lemma (see  \cite[Chapter 2, page 23]{Matt} for more details), we get a
family of disjoint balls, denoted by $B_{i}=B(\xi_{i},c/\sqrt{L})$, such that 
$M$ is covered by the union of $5B_{i}$. This union may be finite or countable. 
Let $\xi\in M$ and consider $B:=B(\xi,
1/\sqrt{L})$. Let $n$ be the number of balls $\bar{B_{i}}$ such that
$\bar{B}\cap \bar{5B_{i}}\neq \emptyset$. Since $\bar{B}$ is compact, we have a
finite number of these balls, so that
\[
\bar{B}\subset \cup^{n}_{i=1}\bar{5B_{i}}.
\]
We claim that $n$ is independent of $L$. Hence we will get our statement because
\[
\mu_{L}(B)\leq \sum^{n}_{i=1}\mu_{L}(B(\xi_{i}, 5c/\sqrt{L}))\lesssim
\frac{n}{k_{L}}.
\]
We only need to check that $n$ is independent of $L$. Using the triangle
inequality, for all $i=1,\ldots,n$
\[
B(\xi_{i}, c/\sqrt{L})\subset B(\xi, 10/\sqrt{L}).
\]
Therefore,
\[
\cup^{n}_{i=1}B(\xi_{i}, c/\sqrt{L})\subset B(\xi,10/\sqrt{L}),
\]
where the union is a disjoint union of balls. Now,
\[
\frac{10^m}{L^{m/2}}\approx \vol(B(\xi,10/\sqrt{L}))\geq
\sum^{n}_{i=1}\vol(B_{i})\approx n\frac{c^m}{L^{m/2}}.
\]
Hence, $n\lesssim (10/c)^m$ and we can choose it independently of $L$.
\end{proof}
Now we can prove the characterization of the Carleson measures.
\begin{proof}[Theorem \ref{charactseqCarlmeas}]
Assume condition \eqref{condCarlmeas}. We need to prove the existence of a
constant $C>0$ (independent of $L$) such that for each $f\in E_{L}$
\[
\int_{M}|f|^{2}d\mu_{L}\leq C\int_{M}|f|^{2}dV.
\]
Let $f\in E_{L}$ with $L$ and $r>0$ (small) fixed.
\begin{align*}
\int_{M}& |f(z)|^{2}d\mu_{L}
\overset{\eqref{estf2}}{\leq}C_{r}L^{(m+1)/2}\int_{M}\int_{B(z,r/\sqrt{L})\times
I_{r}}|h(w,s)|^{2}dV(w)dsd\mu_{L}(z)\\
& =C_{r}L^{(m+1)/2}\int_{M\times
I_{r}}|h(w,s)|^{2}\mu_{L}(B(w,r/\sqrt{L}))dV(w)ds\\
& \leq C_{r}L^{(m+1)/2}\frac{1}{k_{L}}\int_{M\times I_{r}}|h(w,s)|^{2}dV(w)ds
\overset{\text{Proposition \ref{normhandf}}}{\approx}\left\|f\right\|^{2}_{2}
\end{align*}
with constants independent of $L$. Therefore, $\mu=\left\{\mu_{L}\right\}_{L}$
is $L^{2}-$Carleson for $M$.\\
\\
For the converse, assume that $\mu$ is $L^2-$Carleson for $M$. We have to show
the existence of a constant $C$ such that for all $L\geq 1$ and
$\xi\in M$, 
$\mu_{L}(B(\xi,c/\sqrt{L}))\leq C/k_{L}$ (for some small constant $c>0$). We
will argue by contradiction, i.e. assume that for all $n\in \mathbb N$ there
exists $L_{n}$ and a ball $B_{n}$ of radius $c/\sqrt{L_{n}}$ such that
$\mu_{L_{n}}(B_{n})>n/k_{L_{n}}\approx n/L^{m/2}_{n}$ ($c$ will be chosen
later). Let $b_{n}$ be the center of the ball $B_{n}$. Define
$F_{n}(w)=K_{L_{n}}(b_{n},w)$. Note that the function $L^{-m/4}_{n}F_{n}\in
E_{L_n}$ and
$\left\|F_{n}\right\|^{2}_{2}=K_{L_{n}}(b_{n},b_{n})\approx L^{m/2}_{n}$.
Therefore,
\begin{align*}
C& \approx \int_{M}|L^{-m/4}_{n}F_{n}|^{2}dV\gtrsim
\int_{M}|L^{-m/4}_{n}F_{n}|^{2}d\mu_{L_{n}} \gtrsim
\int_{B_{n}}|L^{-m/4}_{n}F_{n}|^{2}d\mu_{L_{n}}\\
& \geq \inf_{w\in B_{n}}|L^{-m/4}_{n}F_{n}(w)|^{2}\mu_{L_{n}}(B_{n}) \gtrsim
\inf_{w\in B_{n}}|F_{n}(w)|^{2}\frac{n}{L^{m}_{n}}.
\end{align*}
Now we will study this infimum. Let $w\in B_{n}=B(b_{n},c/\sqrt{L_{n}})$. Then
\[
|F_{n}(b_{n})|-|F_{n}(w)|\leq |F_{n}(b_{n})-F_{n}(w)|\leq
\frac{c}{\sqrt{L_{n}}}\left\|\nabla F_{n}\right\|_{\infty}\leq
\]
\[
\overset{\text{Proposition
\ref{goodgradestimate}}}{\leq}\frac{c}{\sqrt{L_{n}}}C_{1}\sqrt{k_{L_{n}}}\sqrt{
L_{n}}\left\|F_{n}\right\|_{2}\approx cC_{1}k_{L_{n}}.
\]
We pick $c$ small enough so that
\[
\inf_{B_{n}}|F_{n}(w)|^{2}\geq CL^{m}_{n}.
\]
Finally, we have shown that $C \gtrsim n$ $\forall n$. This gives us the
contradiction.
\end{proof}
The following result is a Plancherel-P\'olya type theorem but in the context of
the Paley-Wiener spaces $E_{L}$.
\begin{thm}[Plancherel-P\'olya Theorem]\label{unionsepfam}
Let $\Z$ be a triangular family of points in $M$, i.e. 
$\Z=\left\{z_{Lj}\right\}_{j\in\left\{1,\ldots,m_{L}\right\},L\geq 1}\subset
M$. Then $\Z$ is a finite union of
uniformly separated families, if and only if there exists a constant $C>0$ such
that for all $L\geq 1$ and $f_{L}\in E_{L}$
\begin{equation}\label{PlancherelPolya}
\frac{1}{k_{L}}\sum^{m_{L}}_{j=1}|f_{L}(z_{Lj})|^{2}\leq
C\int_{M}|f_{L}(\xi)|^{2}dV(\xi).
\end{equation}
\end{thm}
\begin{remark}
The above result is interesting because the inequality \eqref{PlancherelPolya}
means that the sequence of normalized reproducing kernels 
is a Bessel sequence for $E_{L}$, i.e.
\begin{align*}
\sum^{m_L}_{j=1}|<f,\tilde{K}_{L}(\cdot,z_{Lj})>|^2\lesssim
\left\|f\right\|^2_{2}\quad \forall f\in E_{L},
\end{align*}
where $\left\{\tilde{K}_{L}(\cdot,z_{Lj})\right\}_{j}$ are the 
normalized reproducing kernels. 
Note that $|\tilde{K}_{L}(\cdot,z_{Lj})|^2\approx
|K_{L}(\cdot,z_{Lj})|^2 k_{L}^{-1}$. That's the reason why the
quantity $k_{L}$ appears in the inequality \eqref{PlancherelPolya}.
\end{remark}
\begin{proof}
This is a consequence of Theorem \ref{charactseqCarlmeas} applied to the
measures
\[
\mu_{L}=\frac{1}{k_{L}}\sum^{m_{L}}_{j=1}\delta_{z_{Lj}}, \ L\geq 1.
\]
\end{proof}
\section{Characterization of Logvinenko-Sereda Sets}
Before we state the characterization, we would like to recall some history of these
kind of inequalities. The classical Logvinenko-Sereda (L-S) theorem describes
some equivalent norms for functions in the Paley-Wiener space $PW^p_{\Omega}$.
The precise statements is the following:
\begin{thm}[L-S]
Let $\Omega$ be a bounded set and $1\leq p<+\infty$. A set $E\subset \R^{d}$
satisfies
\[
\int_{\R^d}|f(x)|^pdx\leq C_{p}\int_{E}|f(x)|^pdx,\ \forall f\in
PW^p_{\Omega},
\]
if and only if there is a cube $K\subset \R^{d}$ such that
\[\inf_{x\in \R^{d}}|(K+x)\cap E|>0.
\]
\end{thm}
One can find the original proof in \cite{LS} and another proof can be found in
\cite[p. 112-116]{HJ}.\\
\\
Luecking in \cite{Lueck} studied this notion for the Bergman spaces. Following
his ideas, in \cite{MarOrt}, it has been proved the following result.
\begin{thm}
Let $1\leq p<+\infty$, $\mu$ be a doubling measure and let $\mathcal
E=\left\{E_{L}\right\}_{L\geq 0}$ be a sequence of sets in $\mathbb S^{d}$. Then
$\mathcal E$ is $L^{p}(\mu)-$L-S if and only if $\mathcal E$ is $\mu-$relatively
dense.
\end{thm}
For the precise definitions and notations see \cite{MarOrt}. Using the same
ideas, we will prove the above theorem for the case of our arbitrary compact
manifold $M$ and the measure given by the volume element.\\
\\
In what follows, $\A=\left\{A_{L}\right\}_{L}$ will be a sequence of sets in
$M$.
\begin{defini}
We say that $\A$ is L-S if there exists a constant $C>0$ such that for any $L$
and $f_{L}\in E_{L}$
\[
\int_{M}|f_{L}|^{2}dV\leq C\int_{A_{L}}|f_{L}|^{2}dV.
\]
\end{defini}
\begin{defini}
The sequence of sets $\A\subset M$ is relatively dense if there exists $r>0$ and
$\rho>0$ such that for all $L$
\[
\inf_{z\in M}\frac{\vol(A_{L}\cap B(z,r/\sqrt{L}))}{\vol(B(z,r/\sqrt{L}))}\geq
\rho>0.
\]
\end{defini}
\begin{remark}
It is equivalent to have this property for all $L\geq L_{0}$ for some $L_{0}$
fixed.
\end{remark}
Our main statement is the following:
\begin{thm}
$\A$ is L-S if and only if $\A$ is relatively dense.
\end{thm}
We shall prove the two implications in the statement separately. First we will
show that this condition is necessary. Before proceeding, we construct functions
in $E_{L}$ with a desired decay of its $L^2$-integral outside a ball.

\begin{prop}\label{construcfunctions}
Given $\xi\in M$ and $\epsilon>0$, there exist functions $f_{L}=f_{L,\xi}\in
E_{L}$ and $R_{0}=R_{0}(\epsilon,M)>0$ such that
\begin{enumerate}
  \item $\left\|f_{L}\right\|_{2}=1$.
  \item 
\[
\int_{M\setminus B(\xi,R_0/\sqrt{L})}|f_{L}|^2dV<\epsilon \quad \forall L.
\]
  \item For all $L\geq 1$ and any subset $A\subset M$,
\[
\int_{A}|f_L|^2dV\leq C_{1}\frac{\vol(A\cap
B(\xi,R_{0}/\sqrt{L}))}{\vol(B(\xi,R_{0}/\sqrt{L}))}+\epsilon,
\]
where $C_{1}$ is a constant independent of $L$, $\xi$ and $f_L$.
\end{enumerate}
\end{prop}
\textbf{Remark.} In the above Proposition, the $R_{0}$ does not depend on the
point $\xi$.
\begin{proof}
Given $z,\xi\in M$ and $L\geq 1$, let $S^{N}_{L}(z,\xi)$ denote the Riesz kernel 
of index $N\in\mathbb N$ associated to the Laplacian, i.e
\[
S^{N}_{L}(z,\xi)=\sum^{k_{L}}_{i=1}\left(1-\frac{\lambda_{i}}{L}\right)^{N}
\phi_{i}(z)\phi_{i}(\xi).
\]
Note that $S^{0}_L(z,\xi)=K_L(z,\xi)$. The Riesz kernel satisfies the following inequality.
\begin{equation}\label{RieszKernelEstimate}
|S^{N}_{L}(z,\xi)|\leq CL^{m/2}(1+\sqrt{L}d(z,\xi))^{-N-1}.
\end{equation}
This estimate has its origins in H\"ormander's article \cite[Theorem
5.3]{HormRiesz}. Estimate \eqref{RieszKernelEstimate} can be found also in
\cite[Lemma 2.1]{Sogge}.\\
Note that on the diagonal, $S^{N}_{L}(z,z)\approx C_{N}L^{m/2}$. The upper
bound is trivial by the definition and the lower bound follows from
\[
 S^N_{L}(z,z)\ge \sum^{k_{L/2}}_{i=1}\left(1-\frac{\lambda_{i}}{L}\right)^{N}
\phi_{i}(z)\phi_{i}(z)\ge 2^{-N}K_{L/2}(z,z)\approx C_{N}L^{m/2}.
\]
Similarly we observe that $\left\|S^{N}_{L}(\cdot,\xi)\right\|^{2}_{2}\approx
C_{N}L^{m/2}$.\\
\\
Given $\xi\in M$, define for all $L\geq 1$
\[
f_{L,\xi}(z):=f_{L}(z)=\frac{S^{N}_{L}(z,\xi)}{\left\|S^{N}_{L}
(\cdot,\xi)\right\|_{2}}.
\]
We will choose the order $N$ later. Each $f_L$ belongs to the space $E_L$ and 
has unit $L^{2}$-norm. Let us verify the second property claimed in 
Proposition \ref{construcfunctions}. Fix a radius $R$. 
Using the estimate \eqref{RieszKernelEstimate},
\[
\int_{M\setminus B(\xi,R/\sqrt{L})}|f_L|^2dV\leq C_{N}L^{m/2} \int_{M\setminus
B(\xi,R/\sqrt{L})}
\frac{dV}{(\sqrt{L}d(z,\xi))^{2(N+1)}}=(\star)
\]
For any $t\geq 0$, consider the following set.
\[
A_{t}:=\left\{z\in M:\quad d(z,\xi)\geq \frac{R}{\sqrt{L}},\quad d(z,\xi)<
\frac{t^{-1/(2(N+1))}}{\sqrt{L}}\right\}.
\]
Note that for $t>R^{-2(N+1)}$, $A_{t}=\emptyset$ and for $t<R^{-2(N+1)}$, 
$A_{t}\subset B(\xi, t^{-1/(2(N+1))}/\sqrt{L})$. Using the distribution function, we have:
\[
(\star)=C_{N}L^{m/2}\int^{R^{-2(N+1)}}_{0}\vol(A_{t})dt \leq
C_{N}\frac{1}{R^{2(N+1)-m}},
\]
provided $N+1>m/2$. Thus if we pick $R_0$ big enough we get
\begin{equation}\label{L2decay}
\int_{M\setminus B(\xi,R_0/\sqrt{L})}|f_L|^2dV<\epsilon.
\end{equation}
Now the third property claimed in Proposition 
\ref{construcfunctions} follows from \eqref{L2decay}. 
Indeed, given any subset $A$ in the manifold $M$,
\[
\int_{A}|f_{L}|^2dV \leq \int_{A\cap B(\xi,R_{0}/\sqrt{L})}|f_L|^2dV+\epsilon.
\]
Observe that
\begin{align*}
\int_{A\cap B(\xi,R_{0}/\sqrt{L})}|f_L|^2dV & \lesssim C_{N}L^{m/2}
\int_{A\cap B(\xi,R_{0}/\sqrt{L})} \frac{dV(z)}{(1+\sqrt{L}d(z,\xi))^{2(N+1)}}\\
& \lesssim C_{N}R^{m}_{0}\frac{\vol(A\cap B(\xi,R_{0}/\sqrt{L}))}
{\vol(B(\xi,R_{0}/\sqrt{L}))}.
\end{align*}
\end{proof}

Now we are ready to prove one of the implications in the characterization of the
L-S sets.
\begin{prop}
Assume $\A$ is L-S. Then it is relatively dense.
\end{prop}
\begin{proof}
Assume $\A$ is L-S, i.e.
\[
\int_{M}|f_{L}|^{2}dV\leq C\int_{A_{L}}|f_{L}|^{2}dV.
\]
Let $\xi\in M$ be an arbitrary point. Fix $\epsilon>0$ and
consider the $R_{0}$ and the functions $f_{L}\in E_{L}$ given by Proposition
\ref{construcfunctions}. Using the third property of Proposition
\ref{construcfunctions} for the sets $A_L$, we get for all $L\geq 1$,
\begin{align*}
1& = \left\|f_{L}\right\|^2_{2}\leq C\int_{A_{L}}|f_{L}|^{2}\leq
CC_{1}\frac{\vol(A_{L}\cap
B(\xi,R_{0}/\sqrt{L}))}{\vol(B(\xi,R_{0}/\sqrt{L}))}+C\epsilon,
\end{align*}
where $C_{1}$ is a constant independent of $L$, $\xi$ and $f_L$.
Therefore, we have proved that there exist constants $c_{1}$ and $c_{2}$ such
that
\[
\frac{\vol(A_{L}\cap B(\xi,R_{0}/\sqrt{L}))}{\vol(B(\xi,R_{0}/\sqrt{L}))}\geq
c_{1}-c_{2}\epsilon.
\]
Hence, $\A$ is relatively dense provided $\epsilon>0$ is small enough.
\end{proof}
Before we continue, we will prove a result concerning the uniform limit of
harmonic functions with respect to some metric.
\begin{lemma}\label{uniformlimit}
Let $\left\{H_n\right\}_{n}$ be a family of uniformly bounded real functions
defined on the ball $\mathbb B(0,\rho)\subset \R^{d}$ for some $\rho>0$. Let $g$
be a non-singular $\mathcal C^{\infty}$ metric such that $g$ and all its
derivatives are uniformly bounded and $g_{ij}(0)=\delta_{ij}$. Define
$g_n(z)=g(z/L_n)$ (the rescaled metrics) where $L_n$ is a sequence of values
tending to $\infty$ as $n$ increases. Assume that the family
$\left\{H_n\right\}_n$ converges uniformly on compact subsets of $\mathbb
B(0,\rho)$ to a limit function $H:\mathbb B(0,\rho)\to\R$ and $H_n$ is harmonic
with respect to the metric $g_n$ (i.e. $\Delta_{g_n}H_n=0$). Then, the limit
function $H$ is harmonic in the Euclidean sense.
\end{lemma}
\begin{proof}
Let $\varphi\in\mathcal C^{\infty}_{c}(\mathbb B(0,\rho))$. 
We have
\[
\int_{\mathbb B(0,\rho)}\Delta_{g}f\varphi dV=\int_{\mathbb
B(0,\rho)}f\Delta_{g}\varphi dV.
\]
It is a direct computation to see that $\Delta_{g_{n}}\varphi\to
\Delta
\varphi$ uniformly and $\Delta_{g_{n}}\varphi$ is uniformly bounded
on $\mathbb B(0,\rho)$. Then
\[
0=\int_{\mathbb B(0,\rho)}H_{n}\Delta_{g_n}\varphi dV_{g_{n}}\to \int_{\mathbb
B(0,\rho)}H\Delta \varphi dm(z)=\int_{\mathbb B(0,\rho)}\Delta H \varphi
dm(z).
\]
Therefore, the limit function $H$ is harmonic in the weak sense. Applying Weyl's
lemma, $H$ is harmonic in the Euclidean sense.\\
\end{proof}
\begin{remark}
The above argument also holds if we have a sequence of metrics $g_n$ converging
uniformly to $g$ whose derivatives also converge uniformly to the derivatives of
$g$. In this case, the conclusion would be that the limit is harmonic with
respect to the limit metric $g$.
\end{remark}
Now, we shall prove the sufficient condition of the main result.
\begin{prop}
If $\A$ is relatively dense then it is L-S.
\end{prop}
\begin{proof}
Fix $\epsilon>0$ and $r>0$. Let $D:=D_{\epsilon,r,f_{L}}$ be
\[
D=\bigl\{z\in M:\ |f_{L}(z)|^2=|h_{L}(z,0)|^2\geq
\epsilon\fint_{B(z,\frac r{\sqrt{L}})\times
I_{r}}\!\!\!|h_{L}(\xi,t)|^2dV(\xi)dt\bigr\}.
\]
The norm of $f_{L}$ is almost concentrated on $D$ because
\begin{align*}
\int_{M\setminus D}& |f_{L}(z)|^2dV(z)\lesssim\\
&\lesssim\epsilon\frac{1}{l(I_{r})}\int_{M\times
I_{r}}|h_{L}(\xi,t)|^2\frac{L^{m/2}}{r^{m}}\int_{(M\setminus D)\cap
B(\xi,r/\sqrt{L})}dV(z)dV(\xi)dt\\
& \lesssim \epsilon\frac{1}{l(I_{r})}\int_{M\times
I_{r}}|h_{L}(\xi,t)|^2dV(\xi)dt\overset{Proposition
\ref{normhandf}}{\lesssim}e^{2r}\epsilon\int_{M}|f_{L}|^2dV.
\end{align*}
It is enough to prove
\begin{equation}\label{aimLSestimateonA}
\int_{D}|f_{L}|^2dV\lesssim \int_{A_{L}}|f_{L}|^2
\end{equation}
with constants independent of $L$ and for this, it is sufficient to show that
there exists a constant $C>0$ such that for all $w\in D$
\begin{equation}\label{aimLSnormonA}
|f_{L}(w)|^2\leq \frac{C}{\vol(B(w,r/\sqrt{L}))}\int_{A_{L}\cap
B(w,r/\sqrt{L})}|f_{L}(\xi)|^2dV(\xi).
\end{equation}
Because then, \eqref{aimLSestimateonA} follows by integrating 
\eqref{aimLSnormonA} over $D$.
So we need to prove \eqref{aimLSnormonA}. Assume it is not true. This means that
for all $n\in \mathbb N$ there exists $L_{n}$, functions $f_{n}\in E_{L_{n}}$
and $w_{n}\in D_{n}=D_{\epsilon,r,f_{n}}$ such that
\[
|f_{n}(w_{n})|^2>\frac{n}{\vol(B(w_{n},r/\sqrt{L_{n}}))}\int_{A_{L_{n}}\cap
B(w_{n},r/\sqrt{L_{n}})}|f_{n}|^2dV.
\]
By the compactness of $M$, there exists $\rho_{0}=\rho_{0}(M)>0$ such that for
all $w\in M$, the exponential map, $\exp_{w}:\mathbb B(0,\rho_{0})\to
B(w,\rho_{0})$, is a diffeomorphism and $(B(w,\rho_{0}),\exp_{w}^{-1})$ is a
normal coordinate chart, where $w$ is mapped to $0$ and the metric $g$ verifies
$g_{ij}(0)=\delta_{ij}$.\\
For all $n\in \mathbb N$, take $\exp_{n}(z):=\exp_{w_{n}}(rz/\sqrt{L_{n}})$
which is defined in $\mathbb B(0,1)$ and act as follows:
\begin{align*}
\exp_{n}:\ & \mathbb B(0,1)\longrightarrow \mathbb
B(0,r/\sqrt{L_{n}})\longrightarrow B(w_{n},r/\sqrt{L_{n}})\\
& \quad z \qquad \longrightarrow \qquad \frac{rz}{\sqrt{L_{n}}} \quad
\longrightarrow \exp_{w_{n}}(rz/\sqrt{L_{n}})=:w
\end{align*}
Consider $F_{n}(z):=f_{n}(\exp_{n}(z)):\mathbb B(0,1)\to
B(w_{n},r/\sqrt{L_{n}})\overset{f_{n}}{\to}\R$ and the corresponding harmonic
extension $h_{n}$ of $f_{n}$. Set
\[
H_{n}(z,t):=h_{n}(\exp_{n}(z),rt/\sqrt{L_{n}}),
\]
 defined on $\mathbb B(0,1)\times J_{1}$ (where $J_{1}=(-1,1)$).\\
Let $\mu_{n}$ be the measure such that
$d\mu_{n}(z)=\sqrt{|g|(\exp_{w_{n}}(rz/\sqrt{L_{n}}))}dm(z)$.
Note that
\[
\int_{B(w_{n},\frac{r}{\sqrt{L_{n}}})} |f_{n}|^2dV=
\frac{r^{m}}{L^{m/2}_{n}}\int_{\mathbb B(0,1)}|F_{n}(z)|^2d\mu_{n}(z).
\]
Therefore, we have
\[
\fint_{B(w_{n},r/\sqrt{L_{n}})}|f_{n}|^2dV\approx\int_{\mathbb
B(0,1)}|F_{n}|^2d\mu_{n}.
\]
We will normalize $H_{n}$ so that
\[
\int_{\mathbb B(0,1)\times J_{1}}|H_{n}(w,s)|^2d\mu_{n}(w)ds=1.
\]
As $w_{n}\in D_{n}$, we have
\[
|F_{n}(0)|^2=|f_{n}(w_{n})|^2\geq \epsilon\fint_{B(w_{n},r/\sqrt{L_{n}})\times
I_{r}}|h_{n}(w,t)|^2dVdt
\]
\[
\approx\epsilon \int_{\mathbb B(0,1)\times
J_{1}}|H_{n}(w,s)|^2d\mu_{n}(w)ds=\epsilon.
\]
Since $|h_{n}|^2$ is subharmonic,
\[
|F_{n}(0)|^2=|h_{n}(w_{n},0)|^2\lesssim \int_{\mathbb
B(0,1)\times J_{1}}|H_{n}(w,s)|^2d\mu_{n}(w)ds=1.
\]
Hence, we have $\forall n\in\mathbb N\quad 0<\epsilon\lesssim
|F_{n}(0)|^2\lesssim 1$.\\
\\
Using the assumption,
\[
\frac{1}{n}\gtrsim \frac{1}{\vol(B(w_{n},r/\sqrt{L_{n}}))}\int_{A_{L_{n}}\cap
B(w_{n},\frac{r}{\sqrt{L_{n}}})}|f_{n}|^2dV\approx \int_{B_{n}\cap \mathbb
B(0,1)}|F_{n}|^2d\mu_{n},
\]
where $B_{n}$ is such that $\exp_{n}(B_{n}\cap \mathbb B(0,1))=A_{L_{n}}\cap
B(w_{n},r/\sqrt{L_{n}})$. So we have that
\[
\begin{cases}
	\forall n\ 0<\epsilon\lesssim|F_{n}(0)|^2\lesssim 1\\
	\forall n\ \int_{\mathbb B(0,1)\cap B_{n}}|F_{n}|^2d\mu_{n}\lesssim
\frac{1}{n}
\end{cases}.
\]
In fact, $0<\epsilon\lesssim |H_{n}(0,0)|^2\lesssim 1$ (by the definition) and
one can prove that $|H_{n}|^2\lesssim 1$. Indeed, if $(z,s)\in\mathbb
B(0,1/2)\times J_{1/2}$, let $w=\exp_{n}(z)\in B(w_{n},r/(2\sqrt{L_{n}}))$ and
$t=rs/\sqrt{L_{n}}\in I_{r/2}$. Then
\begin{align*}
|H_{n}(z,s)|^2 & =|h_{n}(w,t)|^2\lesssim \fint_{B(w,r/(2\sqrt{L_{n}}))\times
I_{r/2}(t)}|h_{n}|^2\\
& \lesssim \fint_{B(w_{n},r/\sqrt{L_{n}})\times I_{r}}|h_{n}|^2dVdt=1.
\end{align*}
Therefore, working with $1/2$ instead of $1$ we have $|H_{n}|^2\lesssim 1$ for
all $n$.\\
\\
The sequence $\left\{H_{n}\right\}_{n}$ is equicontinuous in $\mathbb
B(0,1)\times J_{1}$. Indeed, consider $(w,t)\in B(w_{n},r/(4\sqrt{L_{n}}))\times
I_{r/4}$  and $(\tilde{w},\tilde{t})\in B(w,\tilde{r}r/\sqrt{L_{n}})\times
I_{\tilde{r}r}(t)$, then there exists some small $\delta>0$ such that
\[
|h_{n}(w,t)-h_{n}(\tilde{w},\tilde{t})|\leq
\frac{\tilde{r}}{\sqrt{L_{n}}}r\sup_{B(w,\delta/\sqrt{L_{n}})\times
I_{\delta}(t)}|\nabla h_{n}|\leq (\star)
\]
Taking $\tilde{r}$ small enough so that $\delta\leq r/4$ and 
using Schoen and Yau's estimate \eqref{SchYauestimate}, we have
\[
(\star)\leq \frac{\tilde{r}r}{\sqrt{L_{n}}}\sup_{B(w_{n},r/(2\sqrt{L_{n}}))\times
I_{r/2}}|\nabla h_{n}|
\lesssim\frac{\tilde{r}r}{\sqrt{L_{n}}}\frac{1}{\frac{r}
{\sqrt{L_{n}}}}\sup_{B(w_{n},r/\sqrt{L_{n}})\times I_{r}}|h_{n}|\lesssim
\tilde{r}.
\]
So we have proved that $|h_{n}(w,t)-h_{n}(\tilde{w},\tilde{t})|\leq
C\tilde{r}$. Take $\tilde{r}$ small enough so that $C\tilde{r}<\epsilon$.
Let
$(z,s)\in \mathbb B(0,1/4)\times J_{1/4}$ and $(\tilde{z},\tilde{s})\in \mathbb
B(z,\tilde{r})\times (s-\tilde{r},s+\tilde{r})$. Consider $w=\exp_{n}(z)$,
$t=rs/\sqrt{L_{n}}$, $\tilde{w}=\exp_{n}(\tilde{z})$ and
$\tilde{t}=r\tilde{s}/\sqrt{L_{n}}$. We have proved that for all $\epsilon>0$
there
exists $\tilde{r}>0$ (small) such that for all $(z,s)\in\mathbb B(0,1/4)\times
J_{1/4}$:
\[
|H_{n}(z,s)-H_{n}(\tilde{z},\tilde{s})|<\epsilon \ \text{ if  }\
|z-\tilde{z}|<\tilde{r},\ |s-\tilde{s}|<\tilde{r} \quad \forall n.
\]
Change $1/4$ by $1$. So the sequence $H_{n}$ is equicontinuous.\\
\\
Hence the family $\left\{H_{n}\right\}_{n}$ is equicontinuous and uniformly
bounded on $\mathbb B(0,1)\times J_{1}$. Therefore, by Ascoli-Arzela's theorem
there exists a partial sequence (denoted as the sequence itself) such that
$H_{n}\to H$ uniformly on compact subsets of $\mathbb B(0,1)\times J_{1}$. Since
$F_{n}(z)=H_{n}(z,0)$, we get a function $F(z):=H(z,0):\mathbb B(0,1)\to\R$
which is the limit of $F_{n}$ (uniformly on compact subsets of $\mathbb
B(0,1)$).\\
\\
By hypothesis, the sequence $\left\{A_{L}\right\}_{L}$ is relatively dense.
Taking into account that
$\vol(B(w_{n},r/\sqrt{L_{n}}))=\frac{r^m}{L_{n}^{m/2}}\mu_{n}(\mathbb B(0,1))$,
we get that
\begin{equation}\label{hypothreldense}
\inf_{n}\mu_{n}(B_{n})\geq \rho>0,
\end{equation}
where we have denoted $B_{n}\cap \mathbb B(0,1)$ by $B_{n}$.\\
Let $\tau_{n}$ be such that $d\tau_{n}=\chi_{B_{n}}d\mu_{n}$. From a standard
argument ($\tau_{n}$ are supported in a ball), we know the existence of a weak
*-limit of a subsequence of $\tau_{n}$, denoted by $\tau$. This subsequence will
be noted as the sequence itself. Using \eqref{hypothreldense}, 
we know that $\tau$ is not identically 0. Now we have that
\[
\int_{\mathbb B(0,1)}|F|^2d\tau=0.
\]
Therefore, $F=0$ $\tau-$a.e. in $\mathbb B(0,1)$. Now for all $K\subset \mathbb
B(0,1)$ compact
\[
\int_{K}|F|^2d\tau=0,
\]
therefore $F=0$ in supp$\tau$. Let $\overline{\mathbb B(a,s)}\subset \mathbb
B(0,1)$ such that $\overline{\mathbb B(a,s)}\cap \text{supp}\tau\neq \emptyset$,
then using the fact $B_{n}\subset \mathbb B(0,1)$,
\[
\tau_{n}(\overline{\mathbb B(a,s)})\leq \int_{\overline{\mathbb
B(a,s)}}d\mu_{n}\approx
\frac{L_{n}^{m/2}}{r^{m}}\vol(B(\exp_{n}(a),sr/\sqrt{L_{n}}))\approx s^{m}.
\]
Therefore $\tau_{n}(\overline{\mathbb B(a,s)})\lesssim s^{m}$ for all $n$.
Hence, in the limit $\tau(\overline{\mathbb B(a,s)})\lesssim s^{m}$. In short,
\begin{enumerate}
 	\item We have sets $B_{n}\subset \mathbb B(0,1)$ such that 
  \[ \rho\leq
  \mu_{n}(B_{n})\leq \mu_{n}(\mathbb B(0,1))\approx 1.
  \]
 	\item We have measures $\tau_{n}$ weakly-* converging to $\tau$ 
	(not identically 0).
 	\item $\tau(\overline{\mathbb B(a,s)})\lesssim s^{m}$ for all
$\overline{\mathbb B(a,s)}\subset \mathbb B(0,1)$.
 	\item $|F|=0$ $\tau$-a.e. in $\mathbb B(0,1)$.
	\item $|F(0)|>0$ and $|F|\lesssim 1$.
\end{enumerate}
Assume we know that $H$ is real analytic, then $F(z)$ is real analytic. Federer
(\cite[Theorem 3.4.8]{Fed}) proved that the $(m-1)-$Hausdorff measure $\mathcal
H^{m-1}(F^{-1}(0))<\infty$. Hence $\mathcal H^{m-1}(\text{supp}\tau)\leq
\mathcal H^{m-1}(F^{-1}(0))<\infty$. This implies that the Hausdorff dimension
$\dim_{\mathcal H}(\text{supp}\tau)\leq m-1$. On the other hand, since
$\tau(\overline{\mathbb B(a,s)})\lesssim s^{m}$, we have
\[
0<\tau(\text{supp}\tau)\lesssim \mathcal H^{m}(\text{supp}\tau)
\]
and this implies that $\dim_{\mathcal H}(\text{supp}\tau)\geq m$ 
by Frostman's lemma. So we reach to a contradiction
and the proof is complete. We only need to check that $H$ is real analytic. In
fact, we will show that $H$ is harmonic. We have the following properties:
\begin{enumerate}
	\item Observe that the family of measure $d\mu_{n}$ converges uniformly
to the ordinary Euclidean measure because
$g_{ij}(\exp_{w_{n}}(rz/\sqrt{L_{n}}))\to
g_{ij}(\exp_{w_{0}}(0))=\delta_{ij}$, where $w_{0}$ is the limit point of some
subsequence of $w_{n}$ (recall that we are taking normal coordinate charts).
 	\item If $g_{n}(z):=g(rz/\sqrt{L_{n}})$ (i.e. $g_{n}$ is the rescaled
metric), then $\Delta_{(g_{n},Id)}H_{n}(z,s)=0$ for all $(z,s)\in \mathbb
B(0,1)\times J_{1}$, by construction.
 	\item The functions $H_{n}$ are uniformly bounded and converge
uniformly on compact subsets of $\mathbb B(0,1)\times J_{1}$.
\end{enumerate}
We are in the hypothesis of Lemma \ref{uniformlimit} that guarantees the
harmonicity of $H$ in the Euclidean sense. This concludes the proof of the
proposition.
\end{proof}


\begin{thebibliography}{50}

\bibitem{DF} H. Donnelly, C. Fefferman. Growth and Geometry of
Eigenfunctions of the Laplacian. Analysis and partial differential equations, p.
635-655, Lecture Notes in Pure and Appl. Math., 122, Dekker, New York, 1990.

\bibitem{Fed} H. Federer. Geometric measure theory. Die Grundlehren der
mathematischen Wissenschaften, Band 153 Springer-Verlag New York Inc., New York
1969. 

\bibitem{GreWu} R.E. Greene, H. Wu. Integrals of Subharmonic Functions on
Manifolds of Nonnegative Curvature. Inventiones mathematicae, Volume 27, 1974,
p. 265-298.

\bibitem{HJ} V. Havin, B. Joricke. The uncertainty principle in
harmonic analysis, Springer-Verlag, Berlin, 1994.

\bibitem{Horm} L. H\"ormander.
The spectral function of an elliptic operator. Acta Math. 121 (1968), p.
193-218. 

\bibitem{HormRiesz} L. H\"ormander.
On the Riesz means of spectral functions and eigenfunction expansions for elliptic 
differential operators. Some Recent Advances in the Basic Sciences, Vol. 2 
(Proc. Annual Sci. Conf., Belfer Grad. School Sci., Yeshiva Univ., New York, 
1965-1966),  p. 155-202.

\bibitem{Jost} J. Jost, Riemannian geometry and geometric analysis. Fifth
edition. Universitext. Springer-Verlag, Berlin, 2008. xiv+583 pp. ISBN:
978-3-540-77340-5. 

\bibitem{LiSch} Peter Li, Richard Schoen. $L^{p}$ and mean
value properties of subharmonic functions on Riemannian manifolds. Acta Math.
153 (1984), no. 3-4, p. 279-301. 

\bibitem{LS} V.N. Logvinenko, Ju F. Sereda.
Equivalent norms in spaces of entire functions of exponential type. Teor.
funktsii, funkt. analiz i ich prilozhenia 20, 102-111, 175, 1974.

\bibitem{Lueck} D. H. Luecking. Equivalent norms on $L^p$ spaces of harmonic
functions, Monatsh. Math. 96, no. 2, 133-141, 1983. 

\bibitem{Mar} J. Marzo.
Marcinkiewicz-Zygmund inequalities and interpolation by spherical harmonics.  J.
Funct. Anal.  250  (2007),  no. 2, p. 559-587. 

\bibitem{MarOrt} J. Marzo, J.
Ortega-Cerd\`a. Equivalent norms for polynomials on the sphere.  Int. Math. Res.
Not. IMRN  2008,  no. 5, Art. ID rnm 154, 18p. 

\bibitem{Matt} Pertti Mattila.
Geometry of Sets and Measures in Euclidean Spaces. Fractals and Rectifiability.
Cambridge University Press. 

\bibitem{OrtSal} J. Ortega-Cerd\`a, J. Saludes.
Marcinkiewicz-Zygmund inequalities.  J. Approx. Theory  145  (2007),  no. 2, p.
237-252. 

\bibitem{SchYau} R. Schoen, S.-T. Yau.
Lectures on Differential Geometry. Conference Proceedings and Lecture Notes in
Geometry and Topology, Volume I.

\bibitem{Sogge} Christopher D. Sogge.
On the Convergence of Riesz Means on Compact Manifolds. The Annals of 
Mathematics, Second Series, Vol. 126, no. 3, p. 439-447.

\bibitem{Willmore} T. J. Willmore. Mean value
theorems in harmonic riemannian spaces. J. London Math. Soc.  25, (1950), p.
54-57.

\end{thebibliography}
\end{document}